\documentclass[12pt, 14paper,reqno]{amsart}
\usepackage{amsmath,amsfonts,amssymb}
\usepackage[breaklinks]{hyperref}
\usepackage{graphicx}
\usepackage{longtable}
\usepackage{array}
\usepackage{caption}
\usepackage{enumitem}
\usepackage{tikz}
\usepackage{float}
\usetikzlibrary{positioning}
\makeatletter
\@namedef{subjclassname@2020}{%
  \textup{2020} Mathematics Subject Classification}
\makeatother

\theoremstyle{plain}
\theoremstyle{plain}

\newtheorem{theorem}{Theorem}
\newtheorem{lemma}{Lemma}

\numberwithin{equation}{section}

\newtheorem{prop}{Proposition}[section]

\theoremstyle{proof}

\numberwithin{equation}{section}

\begin{document} 

\title[An infinite family of imaginary biquadratic fields with a large class number]{An infinite family of imaginary biquadratic fields with a large class number}
\author{ P. Rao, K. Chakraborty  and A. Dabhole}
\address{K. Chakraborty, SRM University AP, Amaravati, Andhra Pradesh - 522540.}
\email{kalyan.c@srmap.edu.in}
\address{P. Rao, SRM University AP, Amaravati, Andhra Pradesh--522540}
\email{pratik\_rao@srmap.edu.in}
\address{A. Dabhole, SRM University AP, Amaravati, Andhra Pradesh--522540}
\email{aishwarya\_dabhole@srmap.edu.in}
\maketitle
\begin{abstract} 
We construct an infinite family of imaginary biquadratic fields with large class numbers. The construction of this family is done by composing two distinct families of imaginary quadratic fields: the first one is the family $\mathbb{F}_{l,k}(p)=\mathbb{Q}(\sqrt{k^{2}-p^{l}})$ introduced by Banerjee and Hoque \cite{KB2024}, subject to specific constraints on the integer parameters, $l,k ~\&~p$ and the second one is the parametrized family of the form $\mathbb{Q}(\sqrt{-F_{50s+25}})$ based on Fibonacci numbers, introduced by Kishi \cite{KISHI2008}. One of the main tools that is used is Kuroda's class number formula of imaginary biquadratic fields. We choose these two families with the sole aim to get high order elements in the class group of the resultant biquadratic fields.
\end{abstract}

\maketitle
\section{Introduction }

The study of class groups and class numbers of algebraic number fields has been a key focus in algebraic number theory. Specifically, the challenge of creating infinite families of number fields with class groups that have specific divisibility properties has gained significant interest in recent decades. These investigations connect closely to Hilbert class fields, unramified extensions, the arithmetic of elliptic and hyperelliptic curves.

Amongst the algebraic number fields, quadratic fields occupy a distinguished position because of their arithmetic nature and their rich class group structure. In recent years, considerable attention has been devoted to the study of class groups of quadratic fields, leading to the construction of several infinite families with prescribed class number divisibility. Numerous authors have constructed infinite families of imaginary quadratic fields whose class numbers are divisible by a prescribed integer. We refer to \cite{KB2024}, \cite{KC2018}, \cite{cohn2002},\cite{kishi2011}, \cite{sundar2000} for additional information.

In contrast, the arithmetic of biquadratic fields, especially imaginary biquadratic fields, has received relatively little attention (for more details, see \cite{ath2023}). Sime in \cite{1sime1995} and \cite{2sime1995} investigated the class groups and Hilbert class fields of real biquadratic fields. 
 The literature concerning imaginary biquadratic fields is rather limited. 
 In a recent work, Banerjee et al. \cite{KB2025} have proved that the imaginary biquadratic field
\[
\mathbb{Q}\left(\sqrt{k^2-p^\ell},\sqrt{-q^3-d^2qa+d^3b}\right)
\]
has an element of order $\ell n$ in the class group, for $n\leq 16$ and $n\neq 2$, Here, $k\in\mathbb{Z}$, $\ell\in\mathbb{Z}_{>0}$ such that $gcd(k,~p)=1$, $k^{2}-p^{\ell}<0$ and $p,q$ are
primes. Moreover, $d$ is a positive square-free integer, while
$a,b\in\mathbb{Q}$ are the coefficients of the elliptic curve
\[
E:\quad y^2=x^3+ax+b.
\] This particular family stands out especially because the class number divisibility comes from geometric construction via elliptic curves. 
The lack of literature motivates the search for new infinite families of imaginary biquadratic fields with large class numbers.

In this manuscript, we obtain an infinite family of imaginary biquadratic fields with large class groups using Kuroda's formula \cite{kur} under appropriate hypotheses and by combining class-group-theoretic techniques. Let us begin by recalling the families constructed by Banerjee and Hoque \cite{KB2024}. 

\subsection{The family of Imaginary quadratic fields of Banerjee and Hoque \cite{KB2024}} 

They obtained an infinite family of imaginary quadratic fields with class numbers divisible by a prescribed odd integer. Their result (\cite{KB2024},~ Theorem $5.1$) is stated below:

\begin{prop}\label{prop1}
Assume that $k$ is a positive integer, $\ell \geq 3$ is an odd integer, and $p$ is an odd prime satisfying
\[
\gcd(k,p)=1 \quad \text{and} \quad k^2<p^{\ell}.
\]
Let $-d$ be the square-free part of the integer $k^2-p^{\ell}$, where $d>3$. Let the following conditions hold:
\begin{enumerate}
    \item
    \[
    k \not\equiv \pm 1 \pmod d.
    \]

    \item For every proper divisor $t$ of $k$ and every prime divisor $q$ of $\ell$,
    \[
    t^q \not\equiv \pm k \pmod d.
    \]

    \item Whenever
    \[
    d \equiv 3 \pmod 4, \qquad 3 \mid \ell,
    \]
    and $k'$ is an odd positive divisor of $k$, one has
    \[
    p^{\ell/3} \neq \frac{k'^3+2k}{3}.
    \]
\end{enumerate}
Then the ideal class group of the imaginary quadratic field
\[
\mathbb{F}_{\ell,k}(p)=\mathbb{Q}(\sqrt{-d})
\]
contains an element of exact order $\ell$.
\end{prop}

For ease, we set
\[
\mathcal{P}
=
\{
p \ |
\begin{array}{l}
p \text{ is an odd prime satisfying the hypotheses of Proposition~\ref{prop1}}
\end{array}
\}.
\] 
We will use this notation later. Now we shall recall the family due to Kishi \cite{KISHI2008}.

\subsection{The family of imaginary quadratic fields of Kishi \cite{KISHI2008}} 

 Kishi in \cite{KISHI2008} investigated arithmetic properties of imaginary quadratic fields whose discriminant consists of Fibonacci numbers. To state Kishi's result, we prepare some notations. Let $(F_n)$ and $(L_n)$ be the Fibonacci and Lucas sequences, respectively. These numbers are defined recusively as follows:
\[
F_{1}=1, ~F_{2}=1, ~F_{n+2}=F_{n+1}+F_{n} \qquad(n\geq 1).
\]
\[
L_{1}=1, ~L_{2}=3, ~L_{n+2}=F_{n+1}+F_{n} \qquad(n\geq 1).
\]
In particular, he proved the following theorem (\cite{KISHI2008}, Theorem 1).

\begin{prop}\label{prop2}
For every positive integer $s$, the class number of the quadratic field $\mathbb{Q}(\sqrt{-F_{50s+25}})$ is divisible by $5$, $F_n$ denotes the $n^{th}$ Fibonacci number.
\end{prop}
The first result of this manuscript is a minor improvement of the above Proposition \ref{prop2} of Kishi, by showing that actually $10$ divides the class number of a subfamily of $\mathbb{Q}(\sqrt{-F_{50s+25}})$, where $s\geq 1$. This observation may be of independent interest.
\begin{theorem}\label{thm1}
The class number of the field $\mathbb{Q}(\sqrt{-F_{50s+25}})$ is divisible by $10$, where $s\geq 1$ is an integer satisfying the following conditions: \\
\hspace*{2.5cm}
$(1.1)$ \hspace*{0.5cm}$v_{5}(2s+1)\equiv1\pmod{2}$; \\
\hspace*{2.5cm}
$(1.2)$ \hspace*{0.5cm}$v_{3001}(2s+1)\equiv0\pmod{2}$.
\end{theorem}
Now we move into the main result of this manuscript, i.e., the construction of certain imaginary biquadratic fields with large class number. The following result shows that the two distinct families (that is, the quadratic fields contained inside the two families) are pairwise linearly disjoint.
\begin{theorem}\label{thm3}
Let $k$ and $\ell$ be positive integers with $\ell$ odd. Let \(s\) be a positive integer satisfying
$$
s\not\equiv 1\pmod 3.
$$
Set
$$
M=F_{50s+25}.
$$
Then there exist infinitely many primes \(p\) such that
$$
\gcd(k^2-p^\ell,M)=1.
$$
\end{theorem}

\begin{theorem}\label{thm4}
Let $k$ and $\ell$ be positive integers satisfying the hypotheses of Proposition \ref{prop1}. For every $p\in\mathcal{P}$ obtained by Theorem \ref{thm3}~, the imaginary
biquadratic field
\[
K=\mathbb{Q}\left(\sqrt{k^2-p^\ell},\sqrt{-M}\right)
\]
has class number divisible by at least $5\ell$ or $10\ell$ depending on the unit index $q=1$ or $2$.
\end{theorem}
We actually work with a subfamily of Kishi's family. The next result proves the infinitude of this subfamily..
\begin{theorem}\label{thm2}
    The family 
    \[
    \{\mathbb{Q}(\sqrt{-F_{50s+25}}): ~s\in \mathbb{Z}^{+},~ s\not\equiv 1\pmod{3} \},
    \]
    contains infinitely many distinct imaginary quadratic fields.
\end{theorem}

The manuscript is organized as follows. In \S 2, we recall the necessary preliminaries concerning Kuroda's class number formula, Fibonacci numbers, and their properties. \S 3 is devoted to proving the auxiliary lemmas required for the construction. In \S 4, we prove the main theorems \ref{thm1}~, \ref{thm2}~, \ref{thm3} and \ref{thm4} and construct the desired family of imaginary biquadratic fields. The main tool used here is Kuroda's class number formula for imaginary biquadratic fields. The following notations will be followed throughout. 
\begin{align*}
v_{p}(n) &: \text{ The $p$-adic valuation of an integer $n$, where $p$ is a rational prime} .\\
\left(\frac{r}{s}\right) &: \text{ Legendre symbol of $r$ over $s$, where $r,s \in \mathbb{Z}$ and $s$ is an odd prime }. \\
\left(a,b\right) &: \text{gcd of two integers a and b}. \\
L_{n} &: \text{ $n^{th}$ Lucas number}. \\
F_{n} &: \text{ $n^{th}$ Fibonacci number}.\\
sf(n) &: \text{ Squarefree part of $n$}.
\end{align*}

\section{Preliminaries}
We begin by recalling Kuroda's class number formula due to Lemmermeyer (\cite{kur}, Theorem 1) for an imaginary biquadratic field. Let $k$ be a number field and $K/k$ is a $V_4$-extension, i.e., a normal
extension with $\operatorname{Gal}(K/k)=V_4$, where $V_4$ is the Klein
four-group. The extension $K/k$ has three intermediate fields, say
$k_1$, $k_2$, and $k_3$. The result is as follows:

\begin{prop}\label{kuroda}
Let $K/k$ be a $V_4$-extension of number fields. Then Kuroda's class
number formula holds:
\begin{equation*}
h(K)=2^{\,d-\kappa-2-v}\,
q(K)\,
\frac{h_1h_2h_3}{h_k^2},
\end{equation*}
where $h_i$ $(i=1,2,3)$ denote the class numbers of the three
intermediate quadratic fields, $h_k$ is the class number of $k$, and
$q(K)$ is the unit index of $K/k$. Here $q(K) = (E_K : E_1E_2E_3)$, where $E_i$ is the unit groups of $k_i$ for $1\leq i\leq 3$ and $E_K$ is the unit group of $K$.

In particular,
\[
h(K)=
\begin{cases}
\dfrac{1}{4}\,q(K)\,h_1h_2h_3,
& \text{if } k=\mathbb{Q} \text{ and } K \text{ is real},\\[1.2ex]

\dfrac{1}{2}\,q(K)\,h_1h_2h_3,
& \text{if } k=\mathbb{Q} \text{ and } K \text{ is complex},\\[1.2ex]

\dfrac{1}{4}\,q(K)\,\dfrac{h_1h_2h_3}{h_k^2},
& \text{if } k \text{ is a complex quadratic extension of }
\mathbb{Q}.
\end{cases}
\]
\end{prop}

Now we list some properties of Fibonacci numbers and one important relation between Fibonacci and Lucas numbers. We refer the reader to \cite{koshy} for more details.

\begin{enumerate}[label=(2.P\arabic*)]
    \item $F_{n}$ is odd iff $n$ is not divisible by $3$.
    \item For a positive integer $m$, 
    \[
    5^{2}\mid F_{m} \iff 5^{2}\mid m.
    \].
    \item $F_{n}$ is a perfect square iff $n=1,2,12$.
    \item Let $n$ and $m$ are positive integers, then 
    \[
    (F_{n}, F_{m})=F_{(n,m)}.
    \]
    \item $L_{n}^{2}-5F_{n}^{2}=5(-1)^{2}$.
\end{enumerate}

\textbf{Remark A.} Property (2.P3) is proved by J.H.E. Cohn \cite{cohn}, while the identity (2.P5) was discovered in $1950$ by P. Schub. 

We also need the $p$-adic valuation of Fibonacci numbers. Lengyel \cite{len1995} considered the sequence of Fibonacci numbers $\{F_{n}\}_{n\geq 1}$ and proved the following:
\begin{prop}
    For each positive integer $s$ and each prime number, $p\neq 1, 5,$ we have  

\[
v_{2}(F_n)=
\begin{cases}
0 & \text{if } n \equiv 1, 2\pmod{3},\\

1 & \text{if } n \equiv 3\pmod{6}, \\

3 & \text{if } n \equiv 6\pmod{12},  \\

v_{2}(n)+2 & \text{if } n \equiv 0\pmod{12};
\end{cases}
\]

\hspace*{3.5cm}
$v_{5}(F_n)= v_{5}(n)$;

\[
\hspace*{1.5cm}
v_{p}(F_n)=
\begin{cases}
    v_{p}(n)+v_{p}(F_{z(p)}) & \text{if } n\equiv 0\pmod{z(p)}, \\

    0 & \text{if } n\not\equiv 0\pmod{z(p)}; \\
\end{cases}
\]
where $z(p)$ is the least positive integer such that $p\mid F_{z(p)}$, also known as 'the rank of apparition'.
\end{prop}

\section{Few Lemmas}
We begin by proving some lemmas that are required for proving the aforementioned results.
\begin{lemma}\label{lemma2}
Let $M=F_{50s+25}$ for every fixed positive integer $s$ with $s\not\equiv 1\pmod{3}$. Suppose $q_{1}, q_{2}, \cdots, q_{r}$ are distinct prime divisors of $M$. Then there exist infinitely many primes $p$ such that
\[
\left(\frac{p}{q_{i}}\right)= -1 ,
\]
for every $1\leq i\leq r$.
\end{lemma}
\begin{proof}
 $F_{50s+25}$ being a Fibonacci number, is even if and only if its index is divisible by $3$. So it suffices to show that
\[
3\nmid (50s+25).
\]
Indeed,
\[
50s+25\equiv 2s+1 \pmod{3}.
\]
If $3\mid (50s+25)$, then
\[
2s+1\equiv 0 \pmod{3},
\]
which implies
\[
s\equiv 1\pmod{3},
\]
contradicting the hypothesis. Hence $3\nmid(50s+25)$, and therefore $F_{50s+25}$ is odd. Consequently, every prime divisor $q_i$ of $M$ is an odd prime.

For each $i=1,2,\ldots,r$, choose an integer $a_i$ such that
\[
\left(\frac{a_i}{q_i}\right)=-1.
\]
Such a choice is possible since exactly half of the nonzero residue classes modulo an odd prime are quadratic non-residues.

Since the primes $q_1,q_2,\ldots,q_r$ are pairwise coprime, the Chinese Remainder Theorem guarantees the existence of an integer $a$ satisfying
\[
a\equiv a_i \pmod{q_i}, \qquad i=1,2,\ldots,r.
\]
Since
\[
\left(\frac{a_i}{q_i}\right)=-1,
\]
we have $q_i\nmid a_i$ for every $i$. As $a\equiv a_i\pmod{q_i}$, it follows that $q_i\nmid a$ for every $i$. Hence
\[
(a,q_1q_2\cdots q_r)=1.
\]

By Dirichlet's theorem on primes in arithmetic progressions, there exist infinitely many primes $p$ such that
\[
p\equiv a \pmod{q_1q_2\cdots q_r}.
\]
Therefore,
\[
p\equiv a_i \pmod{q_i}, \qquad i=1,2,\ldots,r.
\]
Since the Legendre symbol depends only on the residue class modulo $q_i$, we obtain
\[
\left(\frac{p}{q_i}\right)
=
\left(\frac{a_i}{q_i}\right)
=
-1,
\qquad i=1,2,\ldots,r.
\]
Therefore, there exist infinitely many primes $p$ such that
\[
\left(\frac{p}{q_i}\right)=-1,\qquad 1\le i\le r.
\]
\end{proof}

\begin{lemma}\label{lemma3}
Let $s\ge1$ be an integer satisfying
\[v_{5}(2s+1)\equiv1\pmod{2}
\quad\text{and}\quad
v_{3001}(2s+1)\equiv 0\pmod{2}.
\]
Then sf($F_{50s+25}$) is divisible by at least two distinct prime numbers.
\end{lemma}

\begin{proof}
Let
\[
n=50s+25=25(2s+1).
\]
Since $5\mid n$, the $5$-adic valuation formula for Fibonacci numbers yields
\[
v_{5}(F_n)=v_{5}(n)
         =2+v_{5}(2s+1).
\]
As  $v_{5}(2s+1)$ is odd, $5$ divides  sf($F_n$).

Since 
\[
F_{25}=5^{2}\cdot3001,
\]
and $3001$ is a primitive prime divisor of $F_{25}$, the rank of apparition of $3001$ is
\[
z(3001)=25.
\]
Now appealing to Lengyel's valuation formula, we derive,
\begin{equation*}
v_{3001}(F_n)
=
v_{3001}(F_{25})
+
v_{3001}\!\left(n\right).
\end{equation*}
Hence
\begin{equation*}
v_{3001}(F_n)
=
1
+
v_{3001}\!\left(2s+1\right).
\end{equation*}
Since $v_{3001}(2s+1)\equiv 0\pmod{2}$, the right-hand side of the above identity is odd. Therefore, $3001$ is a divisor of $F_{n}$ with an odd exponent. Hence, it survives in sf($F_{n}$).    
Therefore, the square-free part sf$F_(n)$ is divisible by at least two distinct prime numbers.
\end{proof}

\section{Proof of the Theorems}

We begin by aligning the notation of Kuroda's formula \ref{kuroda} introduced in \S 2 with the fields considered in this paper. In our framework, the base field $k$ is $\mathbb{Q}$ and $K=\mathbb{Q}(\sqrt{k^{2}-p^{\ell}},~\sqrt{-F_{50s+25}})$, subject to specific constraints on the integer parameters $\ell, k, ~p$ and $s$. Certainly, $K/k$ is a $V_4$ extension and has three subfields $k_1, k_2$ and $k_3$ given in the following figure:
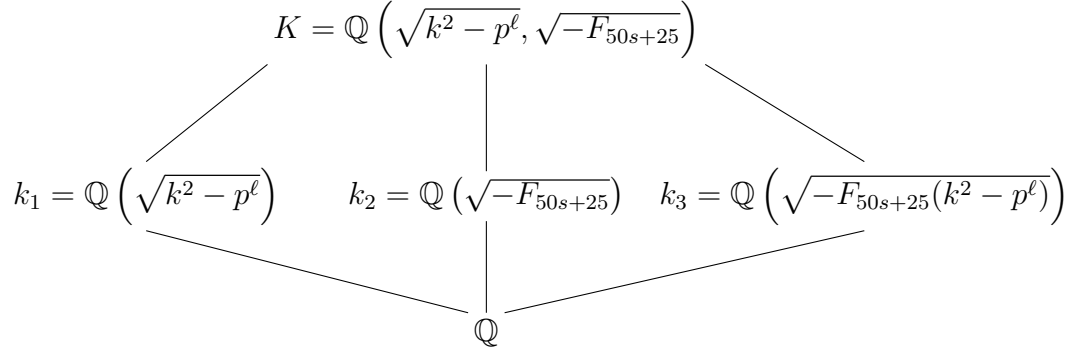
\begin{figure}[H]
\centering
\begin{tikzpicture}[
    every node/.style={inner sep=2pt},
    every path/.style={thin}
]

\node (K) at (0,3.0)
    {$K=\mathbb{Q}\left(\sqrt{k^2-p^\ell},\sqrt{-F_{50s+25}}\right)$};

\node (k1) at (-4.5,0.8)
    {$k_1=\mathbb{Q}\left(\sqrt{k^2-p^\ell}\right)$};

\node (k2) at (0,0.8)
    {$k_2=\mathbb{Q}\left(\sqrt{-F_{50s+25}}\right)$};

\node (k3) at (5.0,0.8)
    {$k_3=\mathbb{Q}\left(\sqrt{-F_{50s+25}(k^2-p^\ell)}\right)$};

\node (Q) at (0,-1.0)
    {$\mathbb{Q}$};

\draw (K.south west) -- (k1.north);
\draw (K.south) -- (k2.north);
\draw (K.south east) -- (k3.north);

\draw (k1.south) -- (Q.north west);
\draw (k2.south) -- (Q.north);
\draw (k3.south) -- (Q.north east);

\end{tikzpicture}
\caption{Subfields of $K/{k}$.}
\label{fig:subfields}
\end{figure}
With these notations in place, we proceed to prove the theorems.

\begin{proof}[Proof of Theorem \ref{thm1}]
   Consider
   \begin{equation*}
       k_{2}=\mathbb{Q}(\sqrt{-F_{50s+25}}),
   \end{equation*}
where $s$ is positive integer given in Theorem \ref{thm1}. By Lemma \ref{lemma3}, it is clear that the discriminant $\Delta_{k_{2}}$ of the quadratic field $k_{2}$ has at least two prime factors. Hence, by Gauss's genus theory (\cite{cox}, Theorem 6.1), the $2-$rank of the class group of $k_{2}$ is at least $1$. Hence $10$ divides the class number of $k_{2}$ as it was shown to be divisible by $5$ by Kishi.
\end{proof}

\begin{proof}[Proof of the Theorem \ref{thm3}]
Let
\[
M=F_{50s+25},
\]
and let
\[
q_1,q_2,\ldots,q_r
\]
be the distinct prime divisors of $M$.
By Lemma \ref{lemma2}~, there exist infinitely many primes $p$ satisfying
\[
\left(\frac{p}{q_i}\right)=-1,
\qquad
1\le i\le r.
\]
We claim that every such prime $p$ satisfies
\[
\gcd(k^2-p^l,M)=1.
\]
Suppose, on the contrary, that
\[
d=\gcd(k^2-p^l,M)>1.
\]
Then $d$ has an odd prime divisor $q$, since $M$ is odd. Moreover,
\[
q\mid M,
\]
so $q=q_i$ for some $i$.
Since
\[
q\mid(k^2-p^l),
\]
we have
\[
p^l\equiv k^2\pmod q.
\]
Hence
\[
\left(\frac{p^l}{q}\right)=1.
\]
Using the multiplicativity of the Legendre symbol, we have 
\[
\left(\frac{p^l}{q}\right)
=
\left(\frac{p}{q}\right)^l.
\]
Since
\[
\left(\frac{p}{q}\right)=-1
\]
by Lemma \ref{lemma2} and $l$ is odd,
\[
\left(\frac{p}{q}\right)^l
=
(-1)^l
=-1,
\]
contradicting the previous equality.
Therefore,
\[
\gcd(k^2-p^l,M)=1.
\]
Since Lemma \ref{lemma2} provides infinitely many such primes $p$, the result follows.
This proves the theorem. 
\end{proof}

\begin{proof}[Proof of the Theorem \ref{thm4}]
Clearly, $(k^{2}-p^{l})M$ has at least two prime divisors for each $p\in \mathcal{P}$ obtained by theorem \ref{thm3}. Hence, by Gauss's genus theory, the $2-$rank of the class group of the real quadratic field $\mathbb{Q}(\sqrt{-(k^{2}-p^{l})M})$ is at least $1$. Therefore, the class number of $\mathbb{Q}(\sqrt{-(k^{2}-p^{l})M})$ is divisible by $2$. Now, by Kuroda's formula, the imaginary biquadratic field $K$ has the class number divisible by at least $5\ell$ or $10\ell$ depending on the unit index $q=1$ or $q=2$. This completes the proof.
\end{proof}

\begin{proof}[Proof of Theorem \ref{thm2}]
We adapt the similar argument as that of Kishi (\cite{KISHI2008},~Theorem 2) used to prove the infinitude of the
family
\[
\left\{
\mathbb{Q}\left(\sqrt{-F_{50s+25}}\right):s\geq 0
\right\}.
\]
Suppose, to the contrary, that
\[
\left\{
\mathbb{Q}\left(\sqrt{-F_{50s+25}}\right)
:\ s\geq 1,\ s\not\equiv 1\pmod{3}
\right\}
\]
is finite. For an integer $n$, denote its square-free part by
$\operatorname{sf}(n)$. Then the set
\[
\mathcal{P}
=
\bigcup_{\substack{s\geq1\\s\not\equiv1\pmod3}}
\left\{
p:\ p\mid\operatorname{sf}(F_{50s+25})
\right\}
\]
is finite. Hence, there exists a positive integer $r$ such that
\[
\mathcal{P}
=
\bigcup_{\substack{1\leq s\leq r\\s\not\equiv1\pmod3}}
\left\{
p:\ p\mid\operatorname{sf}(F_{25(2s+1)})
\right\}.
\]
Let us choose a prime $q>2r+1$ with $q\neq5$. Since $q$ is an odd prime,
\[
q\equiv1\ \text{or}\ 5\pmod6,
\]
and consequently
\[
\frac{q-1}{2}\equiv0\ \text{or}\ 2\pmod3.
\]
Thus
\[
s_0=\frac{q-1}{2}
\]
satisfies $s_0\not\equiv1\pmod3$, and
\[
F_{50s_0+25}=F_{25q}.
\]
Therefore, every prime divisor $d$ of
$\operatorname{sf}(F_{25q})$ belongs to $\mathcal{P}$. Hence, for
some $s$, with
\[
1\leq s\leq r,\qquad s\not\equiv1\pmod3,
\]
we have
\[
d\mid F_{25(2s+1)}.
\]
Since $q>2r+1$ is a prime, we have
\[
(25(2s+1),25q)=25.
\]
It follows from the identity (2.P1) that
\[
\left(F_{25(2s+1)},F_{25q}\right)
=F_{25}
=5^2\cdot3001.
\]
Thus
\[
d\in\{5,3001\}.
\]
Moreover, since $25\mid25q$, by the property~(2.P2), we get
\[
5^2\mid F_{25q}.
\]
Hence $5$ does not divide $\operatorname{sf}(F_{25q})$, and therefore
\[
\operatorname{sf}(F_{25q})=3001.
\]
Thus
\[
F_{25q}=3001A_q^2
\]
for some $A_q\in\mathbb{Z}$.
Now, using the Fibonacci-Lucas identity~(2.P5)
\[
L_n^2-5F_n^2=4(-1)^n
\]
with $n=25q$, we obtain,
\[
L_{25q}^2-5F_{25q}^2=-4.
\]
Consequently,
\[
L_{25q}^2-5\cdot3001^2A_q^4=-4,
\]
or equivalently,
\[
L_{25q}^2
=
5\cdot3001^2A_q^4-4.
\]
Hence
\[
(A_q,L_{25q})
\]
is an integral solution of the quartic equation
\[
Y^2=5\cdot3001^2X^4-4.
\]
Since the Lucas sequence $(L_{n})_{n\geq 1}$ is strictly increasing for $n\geq 1$, the values $L_{25q}$ are distinct as $q$ ranges over primes. On the
other hand, by Siegel's theorem, the above quartic equation has only
finitely many integral solutions. This contradiction proves that
\[
\left\{
\mathbb{Q}\left(\sqrt{-F_{50s+25}}\right)
:\ s\geq1,\ s\not\equiv1\pmod3
\right\}
\]
is infinite.
\end{proof}

\textbf{Remark B.} In theorem \ref{thm4}~, it would be interesting to find arithmetic conditions on the discriminants of $\mathbb{Q}(\sqrt{k^{2}-p^{\ell}})$ and $\mathbb{Q}(\sqrt{-F_{50s+25}})$, to determine when exactly the unit index $q$ is $1$ or $2$. 

\textbf{Concluding Remark.} There are several known families with  class number divisibility for a particular or an arbitrary number $n$. Due to Kuroda's formula and by the construction of the biquadratic field shown in this paper, one can generate many more families of imaginary biquadratic fields with large class numbers. 

\section*{Acknowledgment}

The authors express their gratitude to SRM University-AP for providing the academic environment, resources, and institutional support to carry out this research.


\vspace*{2mm}
\begin{thebibliography}{10}

\bibitem{KB2024} Banerjee, K., Hoque, A., Chow groups, pull back and class groups. Monatsh Math 205, 433–454 (2024). https://doi.org/10.1007/s00605-024-02008-3.

\bibitem{KISHI2008} Kishi, Y., A new family of imaginary quadratic fields whose class number is divisible by five. J. Number Theory 128, 2450–2458 (2008). https://doi.org/10.1016/j.jnt.2008.02.016.

\bibitem{KB2025} Banerjee, K., Chakraborty, K. and Ghosh, A. Class groups of imaginary biquadratic fields. Res. number theory 11, 102 (2025). https://doi.org/10.1007/s40993-025-00685-z.

\bibitem{KC2018} K. Chakraborty, A. Hoque, Y. Kishi, P.P. Pandey,
Divisibility of the class numbers of imaginary quadratic fields,
Journal of Number Theory,
185,
2018,
339-348,
ISSN 0022-314X,
https://doi.org/10.1016/j.jnt.2017.09.007.

\bibitem{cohn2002}  J.H.E. Cohn, On the class number of certain imaginary quadratic fields, Proc. Amer. Math. Soc.
130 (2002) 1275–1277. https://doi.org/10.1090/S0002-9939-01-06255-4.

\bibitem{kishi2011} Ishii, K., On the divisibility of the class number of imaginary quadratic fields, Proc. Japan Acad.
Ser. A 87 (2011) 142–143. https://doi.org/10.3792/pjaa.87.142.

\bibitem{sundar2000}  K. Soundararajan, Divisibility of class numbers of imaginary quadratic fields, J. Lond. Math. Soc.
61 (2000) 681–690. https://doi.org/10.1112/S0024610700008887.

\bibitem{1sime1995} Sime, Patrick J., On the Ideal Class Group of Real Biquadratic Fields. Transactions of the American Mathematical Society 347, 12 (1995): 4855–76. https://doi.org/10.2307/2155066.

\bibitem{2sime1995} Sime, Patrick J.,
Hilbert Class Fields of Real Biquadratic Fields,
Journal of Number Theory,
50, Issue 1,
1995,
154-166,
ISSN 0022-314X,
https://doi.org/10.1006/jnth.1995.1010.

\bibitem{ath2023} Athaide, E., Cardwell, E., and Thompson, C., Class number formulas for certain biquadratic fields. Hardy-Ramanujan Journal 46 (2023),  63-89.  https://doi.org/10.46298/hrj.2024.12573.

\bibitem{len1995} Lengyel, T., The Order of the Fibonacci and Lucas Numbers. The Fibonacci Quarterly, 1995; 33(3), 234–239. https://doi.org/10.1080/00150517.1995.12429139.

\bibitem{cox} Cox, D.A., (2013), Front Matter. In Primes of the Form x2 + ny2, D.A. Cox (Ed.). https://doi.org/10.1002/9781118400722.fmatter.

\bibitem{kur} Lemmermeyer F., Kuroda’s class number formula. Acta Arithmetica 1994; 66: 245–260.  https://doi.org/10.4064/aa-66-3-245-260.

\bibitem{koshy} Koshy, T., Fibonacci and Lucas numbers with application (2nd ed.). John Wiley \& Sons, 2017.

\bibitem{cohn} J. H. E. Cohn, On Square Fibonacci Numbers, J. London Math. Soc., V. 39, Issue 1, 1964, 537–540, https://doi.org/10.1112/jlms/s1-39.1.537.

\bibitem{kuroda}Kuroda S. Über, die Klassenzahlen algebraischer Zahlkörper. Nagoya Mathematical Journal. 1950;1:1-10. doi:10.1017/S0027763000022777.

\end{thebibliography}
\end{document}